\documentclass[11pt,leqno]{amsart}
\usepackage{amssymb,amsmath,amscd,graphicx,fontenc,amsthm,mathrsfs}
\usepackage[pdftex,bookmarks,colorlinks,breaklinks]{hyperref} 
\hypersetup{linkcolor=blue,citecolor=red}
\usepackage{multicol}
\usepackage[numbers]{natbib}
\newtheorem{theorem}{Theorem}[section]
\newtheorem{definition}[theorem]{Definition}
\newtheorem{lemma}[theorem]{Lemma}
\newtheorem{proposition}[theorem]{Proposition}
\newtheorem{corollary}[theorem]{Corollary}
\theoremstyle{definition}
\newtheorem{remark}[theorem]{Remark}

\newcommand{\Sp}{\mathrm{Sp}(V)}
\newcommand{\x}{{\bf x}}
\newcommand{\y}{{\bf y}}
\newcommand{\U}{{\bf U}}
\newcommand{\FF}{\mathcal{F}}
\newcommand{\CC}{\mathbb{C}}
\newcommand{\Q}{\mathbb{Q}}
\newcommand{\ZZ}{\mathbb{Z}}

\newcommand{\mf}{m_{\mathsf{faithful},K}}
\newcommand{\mff}{m_{\mathsf{faithful}}}

\newcommand{\aideal}{\mathfrak{a}}

\newcommand{\1}{\mathbf{1}}
\newcommand{\FZ}{\mathbb{Z}/p\mathbb{Z}}

\newcommand{\rank}{d'}

\newcommand{\Z}[1]{\mathbb{Z}/#1\mathbb{Z}}

\newcommand{\OO}[1]{\mathcal{O}/\mathfrak{p}^{#1}}
\newcommand{\F}[1]{\mathbb{F}_{#1}}
\newcommand{\restr}[2]{#1_{|_{#2}}}

\DeclareMathOperator{\Hom}{Hom}
\DeclareMathOperator{\GL}{GL}

\DeclareMathOperator{\Aff}{Aff}

\DeclareMathOperator{\Hei}{Heis}

\DeclareMathOperator{\Ind}{Ind}

\DeclareMathOperator{\Ann}{Ann}
\usepackage[top=2cm,bottom=2cm,right=2cm,left=2cm]{geometry}
\title{Minimal dimension of faithful representations for $p$-groups}
\author[M. Bardestani]{Mohammad Bardestani}
\address{Mohammad Bardestani, Department of Mathematics and Statistics, University of Ottawa, 585 King Edward, Ottawa, ON K1N
6N5, Canada.}
\email{mbardest@uottawa.ca}
\author[K. Mallahi-Karai]{Keivan Mallahi-Karai}
\address{Keivan Mallahi-Karai, Jacobs University Bremen, Campus Ring I, 28759 Bremen, Germany.}
\email{k.mallahikarai@jacobs-university.de }
\author[H. Salmasian]{Hadi Salmasian}
\address{Hadi Salmasian, Department of Mathematics and Statistics, University of Ottawa, 585 King Edward, Ottawa, ON K1N
6N5, Canada.}
\email{hadi.salmasian@uottawa.ca}
\begin{document}
\maketitle
\begin{abstract}  For a group $G$, we denote by $\mff(G)$, the smallest dimension of a faithful complex representation of $G$. Let $F$ be a non-Archimedean local field with the ring of integers $\mathcal{O}$ and the maximal ideal $\mathfrak{p}$.
  In this paper, we  compute the precise value of $\mff(G)$ when $G$ is the Heisenberg group over $\OO{n}$. We then use the Weil representation to compute the minimal dimension of faithful representations of the group of unitriangular  matrices over $\OO{n}$ and many of its subgroups. By a theorem of Karpenko and Merkurjev~\cite[Theorem 4.1]{Merkurje}, our result yields the precise value of the essential dimension of the latter finite groups.
 \end{abstract}
 \let\thefootnote\relax
\let\thefootnote\relax
\section{Introduction}
\label{Sec:Intro}
For a given finite group $G$ and a field $K$, let $\mf(G)$ denote the least integer $n$ such that $G$ embeds into $\GL_n(K)$. Formally, if $d_{ \rho}$ denotes the degree of a $G$-representation $(\rho,V)$ over $K$, then we define: 
\begin{equation}
\mf(G):=\min_{ \ker \rho= \{1 \} } d_{ \rho}.
\end{equation}
Apart from its intrinsic interest, the question of computing or obtaining reasonable bounds for 
$\mf(G)$ has found several applications. For example when $G={\bf G}(\F{q})$, where $\bf G$ is a Chevalley group, such bounds have proven to be useful in various combinatorial problems that have to do with the group expansion problem~\cite{Green}. 

In contrast, in this paper, we are mostly concerned with situations in which $G$ is a $p$-group. For such groups, $\mf(G)$ is related to the theory of essential dimension of algebraic groups.
The notion of essential dimension $\mathrm{ed}_K(G)$ of a finite group $G$ over a field $K$ was introduced by Buhler and Reichstein~\cite{ReichsteinII}. The integer $\mathrm{ed}_K(G)$ is equal to the smallest number of algebraically independent parameters required to define a Galois $G$-algebra over any field extension of $K$.  Recently, Karpenko and Merkurjev~\cite{Merkurje} proved that the essential dimension of a $p$-group $G$ over a field $K$ containing a primitive $p$th root of unity is equal to $\mf(G)$. 
For brevity of the exposition, we will always assume that $K=\mathbb{C}$ (and hence drop the subscript $K$), though some of our results extend in a straightforward way to arbitrary fields that contain enough roots of unity.

In order to state the results, let us set some notation. 
Throughout this paper $F$ will denote a non-Archimedean local field with discrete valuation $\nu$. We will denote the associated ring of integers by $\mathcal{O}$, the unique maximal ideal of $\mathcal{O}$ by $\mathfrak{p}$, and the associated residue field by $\F{q}$ where $q=p^f$, and we call $f$ the {\it absolute inertia degree} of $F$. Moreover the number $e=\nu(p)$ is called the {\it absolute ramification index} of $F$. We remark that for the local field $\F{q}((T))$, we have $e=\nu(p)=\infty$.  Our first theorem gives 
an explicit formula for the dimension of the smallest faithful representation of the Heisenberg group over a large class of rings.
\begin{theorem}\label{Heisenberg} Let $\Hei_{2k+1}(\OO{n})$ denote the Heisenberg group defined over $\OO{n}$. Then
\begin{equation}
\mff\left(\Hei_{2k+1}(\OO{n})\right)=\sum_{i=0}^{ \xi-1 } fq^{k(n-i)}\ ,\ \xi=\min \left\{ e,n \right\},
\end{equation}
where $f$ is the absolute inertia degree, $e$ is the  absolute ramification index, and $q$ is the size of the residue field of $F$. 
\end{theorem}
For $n=1$ and $\mathcal{O}=\ZZ_p$, this theorem was previously known~\cite[Section 2]{ReichsteinIII}. However, the generalization to arbitrary $n$ and $\mathcal{O}$ requires a few new ideas.
One reason is that unlike the case $n=1$ and $\mathcal{O}=\ZZ_p$, in general the center of $\Hei_{2k+1}(\OO{n})$ is not a cyclic subgroup, and hence the faithful representation of minimal dimension is not necessarily irreducible.
Therefore, one cannot easily infer the minimal dimension of faithful representations of this group from the character table. In order to compute $\mff(\Hei_{2k+1}(\OO{n}))$, we have to carefully analyze the characters of  arbitrary representations of this group. 

To wit, we will first have to classify all irreducible representations of Heisenberg groups. This can be done by applying the celebrated  
Stone-von~Neumann Theorem. As it turns out, this method works nicely for finite quotient rings of {\it unramified} extensions of $p$-adic fields. Indeed in this case one can concretely compute the polarizing subgroups of the generic characters. This does not seem to be easy for finite quotient rings of {\it ramified} extensions and the Stone-von~Neumann Theorem is no longer applicable. Instead, we will directly apply the ``Mackey machine" to classify all irreducible representations of Heisenberg groups.  

By invoking the Weil representation along with Theorem~\ref{Heisenberg},  we will also compute $\mff(\U_k(\OO{n}))$ where $\U_{k}(\OO{n})\subseteq \GL_{k}(\OO{n})$ is the group of unitriangular matrices with entries in $\OO{n}$.
\begin{theorem}\label{Unipotent} Let $\Hei_{2k+1}(\OO{n})\subseteq G\subseteq\U_{k+2}(\OO{n})$ be a group and assume that $\mathrm{char}(\mathcal{O}/\mathfrak{p})\neq 2$. Then  
\begin{equation}
\mff(G)=\mff(\Hei_{2k+1}(\OO{n})).
\end{equation}
\end{theorem}
For some other classes of two-step nilpotent groups, one can also use the Stone--von~Neumann Theorem to give a different (and short) proof the following result, which was obtained previously by Meyer and Reichstein~\cite[Theorem 1.4]{ReichsteinI}.
\begin{theorem}\label{2-step}
Let $H$ be a finite two-step nilpotent $p$-group with a cyclic commutator subgroup. Then 
\begin{equation}
\mff(H)=\sqrt{[H:Z(H)]}+\mff(Z(H))-1,
\end{equation}
where $Z(H)$ is the center of $H$.
\end{theorem} 

For some semidirect products, the minimal dimension of faithful representations can be computed by the following proposition whose proof is fairly simple (see Section~\ref{Representation of Linear groups}).   
\begin{proposition}\label{Aff-prop} Let $H$ be a finite group acting (by group automorphisms) on a cyclic group $C = \langle a \rangle$ of order $p^n$, where $p$ is a prime number and $n \ge 1$. Let $Ha$ denote the $H$-orbit of the generator $a\in C$. Then 
\begin{equation}\label{H-action}
\mff(C \rtimes H)\geq |Ha|. 
\end{equation}  
Moreover, equality occurs in~\eqref{H-action} if $C$ is a faithful $H$-module.
\end{proposition}

It is worth noting that if $C$ is not cyclic, $\mff(C \rtimes H)$ could be much smaller than the size of a typical $H$-orbit. For instance, let $C$ be the direct product of $n$ copies of the group $\Z{2}$ and let $H=S_n$ be the symmetric group on $n$ letters acting on $C$ by permuting the factors. It is easy to see that signed permutation matrices provide a faithful $n$-dimensional representation of the group $C \rtimes H$, whereas typical $H$-orbits have size around $ {n \choose n/2}$, which is clearly much larger than $n$. However, for some spacial semi-direct products involving the (non-cyclic) additive group of $\OO{n}$, we can still give some explicit bounds. For a commutative ring $R$, 
we will write $\Aff(R)$ for the affine group $ R \rtimes R^\times$, where $R^{\times}$, the group of units of $R$, acts by multiplication on the additive group $(R,+)$. 
\begin{theorem}\label{Aff} With the above notation we have
\begin{equation}
\mff(\Aff(\OO{n}))=q^n-q^{n-1}.
\end{equation}
\end{theorem}
\begin{remark}
For a finite field $\mathbb{F}_q$, $\mathrm{char}(\F{q})\neq 2$, let $\iota$ be a faithful one-dimensional representation of the cyclic group $\mathbb{F}^\times_{q^2}$ and let $\rho=\rho_\iota$ be the corresponding cuspidal  representation of $\GL_2(\mathbb{F}_q)$ of dimension $q-1$. Using an explicit computation of the character of $\rho$ (see~\cite{shapiro}, $\S$22) one can see that $\chi_\rho(g)=q-1$ if and only if $g$ is the identity matrix. This shows that $\rho$ is a faithful representation of $\GL_2(\mathbb{F}_q)$ (see Lemma~\ref{Ker-rho}) and so $\mff(\GL_2(\mathbb{F}_q))\leq q-1$. But from Theorem~\ref{Aff} we also know that $\mff(\GL_2(\mathbb{F}_q))\geq q-1$. Hence $\mff(\GL_2(\mathbb{F}_q))=q-1$. Moreover from Theorem~\ref{Aff} we have $\mff(\GL_2(\OO{n}))\geq q^{n}-q^{n-1}$. It is of interest to know how sharp this bound is.
\end{remark}

This paper is organized as follows. In Section~\ref{pre}, we will set some notation and gather some information about the representation theory of additive groups. In Section~\ref{pg}, we analyze faithful representations of $p$-groups and state a verion of the Stone--von Neumann Theorem that will be used later. Section~\ref{heisenberg} and Section~\ref{Weil-Sec} are devoted to the proof of Theorem \ref{Heisenberg} and Theorem~\ref{Unipotent}. Finally we will present the proofs of Theorem~\ref{Aff} in Section~\ref{Representation of Linear groups}. 
\section{Preliminary}\label{pre}
In this section we set some notation which will be used throughout this paper. We also recall some basic facts about local fields that can be found in \cite{Neukirch,Local-fields}.
\subsection{Notation}\label{Notation} 
Let $G$ be a group with the identity element $\1$. If $x,y\in G$ then the commutator of $x$ and $y$ is denoted by  $[x,y]:=xyx^{-1}y^{-1}$. The center and the commutator subgroup of $G$ will be denoted, respectively, by $Z(G)$ and $[G,G]$. For a $p$-group $G$, we write 
$
\Omega_1(G):=\{g\in G: g^p=\1\}.
$
The Pontryagin dual of an abelian group $A$, i.e., $\Hom(A, \CC^{\ast})$, will be denoted by $\widehat{A}$. Evidently, when $A$ is an elementary abelian $p$-group,  $\widehat{A}$ is canonically a $\FZ$-vector space. 
 We will use the shorthand ${\bf e}(x):=\exp(2\pi ix)$. We will denote vectors by boldface letters.  Finally let $H\leq G$ be a subgroup and assume that $\rho: H\to \GL(V)$ is a representation. An extension of $\rho$ to $G$ is a representation $\tilde{\rho}: G\to \GL(V)$ such that $\tilde{\rho}|_H=\rho$.  
\subsection{Additive characters of quotient rings of local fields}\label{Local-Fields}
A non-Archimedean local field is a complete field with respect to a discrete valuation that has a finite residue field. By the well-known classification of local fields, any non-Archimedean local field is isomorphic to a finite extension of $\Q_p$ ($p$ is a prime number) or is isomorphic to the field of formal Laurent series $\F{q}((T))$ over a finite field with $q=p^f$ elements~\cite[Chapter II, \S 4, \S 5]{Local-fields}. For a non-Archimedean local field $F$ with  discrete valuation $\nu$, we will denote its ring of integers, its unique prime ideal and its residue field  by $\mathcal{O}$, $\mathfrak{p}$ and $\F{q}$ respectively. 
We will also fix a uniformizer, denoted by $\varpi$. 
For any integer $m\in\ZZ$, write $\mathfrak{p}^m:=\{x\in F: \nu(x)\geq m\}$. 
We have thus the following filtration of $F$:
$$
F\supseteq\cdots\supseteq\mathfrak{p}^{-2}\supseteq \mathfrak{p}^{-1}\supseteq \mathfrak{p}^0=\mathcal{O}\supseteq\mathfrak{p}\supseteq\mathfrak{p}^2\supseteq\cdots \supseteq \{0\}.
$$
For all $m\in \mathbb{Z}$, there exists a natural isomorphism of additive groups $\mathfrak{p}^{m}/\mathfrak{p}^{m+1}\cong\mathcal{O}/\mathfrak{p}$.

Set $R:=\OO{n}$.  An additive character $\psi:R\to \CC^*$ is called {\it primitive} if $\ker\psi$ does not contain a non-zero ideal of $R$. Note that a primitive character exists, or else all characters of $R$ would come from $R/\mathfrak{a}$, for some ideal $\frak{a}$, and a counting argument shows this is not possible. 

Fix a primitive character $\psi: R\to \CC^*$. For $b\in R$ we define 
$$
\psi_b: R\to \CC^*\ ,\ x\mapsto \psi(bx).
$$
It is straightforward to show that $\Phi_1: R\to \widehat{R}$ that maps $b\mapsto \psi_b$, is a group isomorphism.
Now let $\frak{a}$ be an ideal in $R$ with the annihilator $\Ann(\frak{a}):=\{r\in R: r\frak{a}=0\}$. Since the restriction map $\widehat{R}\to \widehat{\frak{a}}$ is a surjective group homomorphism, we deduce that the map
$$\Phi_2: R/\Ann(\frak{a})\to \widehat{\frak{a}}\ , \ b+\Ann(\frak{a})\mapsto \restr{{\psi_b}}{\frak{a}},$$
 is a group isomorphism.
Notice that $\Ann(\mathfrak{p}^m/\mathfrak{p}^n)=\mathfrak{p}^{(n-m)}/\mathfrak{p}^n$, and so we obtain the following lemma.
\begin{lemma}\label{Character-local} For integers $0\leq m\leq n$, every additive character of the abelian group $\mathfrak{p}^m/\mathfrak{p}^n$ is of the form
\begin{equation*}
\psi_{b}: \mathfrak{p}^m/\mathfrak{p}^n\to \CC^* \ ,
\ x\mapsto \psi(bx),
\end{equation*}
for a unique 
$b+\mathfrak{p}^{(n-m)}/\mathfrak{p}^n
\in 
(\OO{n})/(\mathfrak{p}^{(n-m)}/\mathfrak{p}^n)\cong \OO{n-m}$. In particular $\widehat{\mathfrak{p}^m/\mathfrak{p}^n}\cong \OO{(n-m)}$.
\end{lemma}
\section{Faithful representations of some $p$-groups}\label{pg} In this section after reviewing some basic facts on central characters of faithful representations of $p$-group and recalling the Stone--von~Neumann Theorem,  we compute the minimal dimension of faithful representations of two-step nilpotent $p$-groups with a cyclic commutator subgroup. 
\subsection{Central characters of faithful representations of $p$-groups} 

Let $A$ be a finite abelian group. We denote the minimal number of generators of $A$ by $d(A)$. For an exact sequence of abelian groups $0\to A_1 \to A\to A_2\to 0$, we have the following inequalities  
\begin{equation}\label{d(A)}
\max\{d(A_i): i=1,2\}\leq d(A)\leq d(A_1)+d(A_2).
\end{equation} 

The number of invariant factors of $A$ will be denoted by $\rank(A)$. A direct consequence of the elementary divisor theory is the equality $\rank(A)=d(A)$. 
Evidently we have $\mff(A)\leq \rank(A)$. Now for a given faithful representation $\rho: A\to \GL_m(\CC)$, by decomposing $\rho$ into irreducible representations and applying~\eqref{d(A)}, we get $d(A)=\rank(A)\leq \mff(A)$. Hence for a finite abelian group $A$ we have $\mff(A)=d(A)=\rank(A)$.
We will summarize these in the following lemma:
\begin{lemma} For a finite abelian $p$-group $A$ we have
\begin{equation}\label{Tensor-pabelian}
d(A)=\rank(A)=\mff(A)=\dim_{\FZ}(A\otimes_{\ZZ} \FZ)=\dim_{\FZ}(\Omega_1(A)).
\end{equation}
\end{lemma}
Now let $E$ be a finite elementary abelian $p$-group with the canonical $\FZ$-vector space structure. One can verify that every one-dimensional representation $ \chi: E \to \CC^{\ast}$ factors uniquely as $ \chi=  \epsilon\circ \chi_\circ$, where $ \chi_\circ \in \Hom(E, \FZ)$ and the embedding $ \epsilon:\FZ\to \CC^*$ is defined by $ \epsilon(x+p\ZZ)={\bf e}(x/p)$. Hence the $\FZ$-linear map
\begin{equation}\label{iso-Omei}
\widehat{E}\rightarrow \Hom(E,\FZ)\ ,\  \chi\mapsto \chi_\circ,
\end{equation}
provides an isomorphism of $\FZ$-vector spaces between 
$\widehat{E}$ and $\Hom(E, \FZ)$. 
Now, let $H$ be a finite $p$-group. Applying~\eqref{iso-Omei}, we  obtain the $\FZ$-isomorphism 
\begin{equation}\label{iso-Ome}
\Hom(\Omega_1(Z(H)),\CC^*)\rightarrow \Hom(\Omega_1(Z(H)),\FZ).
\end{equation}
Hereafter the $\FZ$-vector space $\mathrm{Hom}(\Omega_1(Z(H)),\CC^*)$ will be denoted by $\widehat{\Omega}_1(Z(H))$.

\begin{remark}\label{fact=Remark}
Recall the standard fact that for a finite $p$-group $H$, any non-trivial normal subgroup of $H$ intersects  $Z(H)$ and hence $\Omega_1(Z(H))$ non-trivially. Therefore a representation of $H$ is faithful if and only if its restriction to $\Omega_1(Z(H))$ is faithful.
\end{remark}
 We recall the following simple lemma. 
\begin{lemma}\label{Linear-map} Let $g, f_1,\dots, f_n$ be linear functionals on a vector space $V$ 
with respective null spaces $N, N_1,\dots, N_n$. Then $g$ is a linear combination of 
$f_1,\dots, f_n$ if and only if $N$ contains the intersection $N_1\cap\dots\cap N_n$.
\end{lemma}
The following observation, due to Meyer and Reichstein~\cite{ReichsteinI}, will play a crucial role in computing the dimension of minimal faithful representations of $p$-groups.
\begin{lemma}\label{central-span-faith} Let $H$ be a finite $p$-group and let $(\rho_i,V_i)_{1\leq i\leq n}$ be a family of irreducible representations of $H$ with  central characters $\chi_i$. Suppose  that $\{\restr{{\chi_i}}{\Omega_1(Z(H))}: 1\leq i\leq n\}$ spans $\widehat{\Omega}_1(Z(H))$. Then $\oplus_{1\leq i\leq n}\rho_i$ is a faithful representation of $H$.  
\end{lemma}
\begin{proof}
Since $\{\restr{{\chi_i}}{\Omega_1(Z(H))}: 1\leq i\leq n\}$ spans $\widehat{\Omega}_{1}(Z(H))$, from Lemma~\ref{Linear-map} and the $\FZ$-isomorphism~\eqref{iso-Ome} we see that $\bigcap_{i=1}^n\ker \restr{{\chi_i}}{\Omega_1(Z(H))}=\{\1\}$. Hence $\oplus_{1\leq i\leq n}\rho_i$ is a faithful representation of $\Omega_1(Z(H))$ and so from Remark~\ref{fact=Remark}, $\oplus_{1\leq i\leq n}\rho_i$ is a faithful representation of $H$. 
\end{proof}
\begin{lemma}\label{Meyer} Let $H$ be a finite $p$-group and let $\rho$ be a faithful representation of $H$ with the minimal dimension. Then $\rho$ decomposes as a direct sum of exactly $r=d(Z(H))$ irreducible representations
\begin{equation}
\rho=\rho_1\oplus\dots\oplus\rho_r.
\end{equation}
Therefore the set of central characters $\{\restr{{\chi_i}}{\Omega_1(Z(H))}: 1\leq i\leq r\}$ is a basis for $\widehat{\Omega}_1(Z(H))$.
\end{lemma} 
\begin{proof}
Let $\rho=\oplus_{1\leq i\leq n}\rho_i$
be the decomposition of $\rho$ with central characters $\chi_i$, $1\leq i\leq n$. Since $\rho$ is a faithful representation and $r=\rank(Z(H))$, we have $n\geq r$. Furthermore $\bigcap_{i=1}^n\ker\chi_i=\{\1\}$, since $\rho$ is faithful. Hence from Lemma~\ref{Linear-map}, Lemma~\ref{central-span-faith} and minimality of $\dim(\rho)$ we conclude that $n=r$ and $\{\restr{{\chi_i}}{\Omega_1(Z(H))}: 1\leq i\leq r\}$ is a basis for $\widehat{\Omega}_1(Z(H))$.
\end{proof}
As an immediate application of Lemma~\ref{Meyer}, we obtain the following result which provides an upper bound for the minimal dimension of a faithful representation of any $p$-group. This upper bound is sharp for $\Hei_{2k+1}(\Z{p^n})$ according to Theorem~\ref{Heisenberg} in the special case $\mathcal{O}=\ZZ_p$.  
\begin{corollary}\label{upper-pgroups} Let $H$ be a finite $p$-group with center $Z(H)$. Let $A$ be a maximal abelian subgroup of $H$. Then $\mff(H)\leq \mff(Z(H))[H:A]$.
\end{corollary}
\begin{proof} Let $\rho$ be a faithful representation with minimal dimension. Then we have the irreducible decomposition $\rho=\rho_1\oplus\dots\oplus\rho_r$, where $r=\rank(Z(H))=\mff(Z(H))$. The inequality will now follow from the fact that the dimension of any irreducible representation of $H$ is at most $[H:A]$ (see~\cite{Serre}, $\S$3.1, Corollary). 
\end{proof}

The following technical lemma will be used in the proof of Theorem~\ref{2-step}.
\begin{lemma}\label{r-1char} Let $C$ be an abelian $p$-group with a  cyclic subgroup $B$. Then $C$ has $r=d(C)$ one-dimensional representations $\chi_1,\dots,\chi_r$ such that $\ker(\chi_1)\cap B=\{\1\}$, $B\subseteq \bigcap_{i=2}^r\ker(\chi_i)$ and $\chi_1\oplus \chi_{2} \oplus \cdots\oplus\chi_r$ is a faithful representation of $C$.
\end{lemma}
\begin{proof}
Assume $r\geq 2$ and $B \neq \{ \1 \}$, as otherwise the lemma is obvious. Since $B$ is a cyclic subgroup and the restriction map $\widehat{C}\to \widehat{B}$ is surjective, $C$ admits a one-dimensional representation $\chi_1$ such that $\ker(\chi_1)\cap B=\{\1\}$. Set $K=\ker(\chi_1)$. We show that $\rank(K)=r-1$. In fact, since $C/K$ is a cyclic group, $\rank(K)$ is either $r$ or $r-1$. In the case $\rank(K)=r$, the natural $\FZ$-linear map $i: \Omega_1(K)\hookrightarrow \Omega_1(C)$ must be an isomorphism, which is impossible since $K\cap B=\{\1\}$ and $B$ has a non-trivial element of order $p$. Therefore $\rank(KB/B)=r-1$. Pick $r-1$ one-dimensional representations $\chi'_2,\dots,\chi'_{r}$ of $KB/B$ such that $\chi'_2\oplus\dots\oplus\chi'_r$ is a faithful representation of $KB/B$ and extend them to $C/B$. Let $\chi_2,\dots,\chi_r$ be the corresponding one-dimensional representations of $C$. The representations $\chi_1,\dots,\chi_r$ satisfy our desired conditions. 
\end{proof}
Using this lemma we prove the following result.
\begin{corollary}\label{Upper-cyclic-comm} Let $H$ be a finite two-step nilpotent $p$-group with a cyclic commutator subgroup. Let $A$ be a maximal abelian subgroup of $H$. Then $\mff(H)\leq [H:A]+\mff(Z(H))-1$.
\end{corollary}
\begin{proof} Let $r=\rank(Z(H))=\mff(Z(H))$. Since $[H,H]$ is a cyclic subgroup of 
$Z(H)$, by Lemma~\ref{r-1char}, we can find $r$ one-dimensional representations $\chi_1,\dots,\chi_r$ of $Z(H)$ such that $\ker(\chi_1)\cap [H,H]=\{\1\}$ and 
$ \chi_{2}, \dots,  \chi_{r}$ vanish on $[H,H]$. Hence each $ \chi_i$, $2 \le i \le r$, defines a representation of $Z(H)/[H,H]$ which can then be extended to a representation of $H/[H,H]$ and consequently to a one-dimensional representation $\overline{\chi}_i$, $2 \le i \le r$ of $H$. Let $ \overline{ \chi}_{1}$ be also an extension of $ \chi_{1}$ to a character of $A$.  We now claim that 
\[  \rho=\mathrm{Ind}_A^H( \overline{ \chi}_{1})\oplus \overline{ \chi}_{2}\oplus\dots\oplus \overline{ \chi}_{r}, \] 
is a faithful representation of $H$ of dimension $[H:A]+r-1$. The faithfulness follows from the fact that the restriction of $ \rho$ to $Z(H)$ is faithful.  
\end{proof}
\subsection{Stone--von~Neumann Theorem}
Let us first recall a version of the Stone--von~Neumann Theorem that will be used in this paper. It is worth mentioning that the Stone-von Neumann Theorem holds in a much broader setting~\cite{Howe}, but this more general theorem will not be needed here.

Let $H$ be a finite two-step nilpotent group. If $A$ is any subgroup of $H$
containing $Z(H)$, we will denote $\bar{A}:= A/Z(H)$. For $x\in H$, we will similarly denote its image in $H/Z(H)$ by
$\bar x$. Notice that any subgroup of $H$ containing $Z(H)$ is a normal subgroup. Let $\chi$ be a one-dimensional representation of $Z(H)$. This defines a skew-symmetric bilinear form on $\bar{H}$ given by 
\begin{equation}\label{symplectic-module}
\langle \bar{x},\bar{y}\rangle:=\chi\left([x,y]\right).
\end{equation}
$\chi$ is called {\it generic} if the above pairing is non-degenerate, that is, if every character of $\bar{H}$ has the form $\bar{x}\mapsto \langle \bar{x},\bar{y}\rangle$, for a unique $\bar{y}\in\bar{H}$. Assuming $\chi$ is generic, we say that a subgroup $A\leq H$ is {\it isotropic} if $\bar{A}\subseteq \bar{A}^\perp$ where $\bar{A}^\perp=\{\bar{x}\in \bar{H}: \langle \bar{x},\bar{a}\rangle=0,\, \forall \bar{a}\in \bar{A}\}$. We say that $A$ is {\it polarizing} if $\bar{A} =\bar{A}^\perp$. For the proof of the following theorem we refer the reader to~\cite[$\S$4.1]{Bump}.
\begin{theorem}[Stone--von~Neumann Theorem]\label{SV} Let $H$ be a finite two-step nilpotent group, and let $\chi$ be a generic character of its center $Z(H)$. Then
there exists a unique isomorphism class of irreducible representations of $H$ with central character $\chi$. Such a representation may be constructed as follows:
choose any polarizing subgroup $A$ of $H$, and let $\tilde{\chi}$ be any extension of $\chi$ to $A$.
Then $\mathrm{Ind}_A^H(\tilde{\chi})$ will be such a representation. 

\end{theorem}
Now let $\chi$ be a one-dimensional representation of $Z(H)$ such that $\ker(\chi)\cap [H,H]=\{\bf 1\}$. Then it is easy to see that $\chi$ is generic and any maximal abelian subgroup of $H$ is a polarizing subgroup. Therefore we have the following corollary of the Stone--von~Neumann Theorem. 
\begin{corollary}\label{cor-SV} Let $H$ be a finite two-step nilpotent group with center $Z(H)$, and let $\chi$ be a one-dimensional representation of $Z(H)$ such that $\ker(\chi) \cap [H, H] = \{\mathbf{1}\}$. Let $\tilde{\chi}$ be any extension of $\chi$ to a maximal abelian subgroup $A$. Then, up to isomorphism $\Ind_{A}^H(\tilde{\chi})$ is the unique irreducible representation of $H$ with central character $\chi$.
\end{corollary}
We also need the following lemma:
\begin{lemma}\label{max-[H:A]} Let $H$ be a finite two-step nilpotent group such that $[H,H]$ is cyclic. Let $A$ be a maximal abelian subgroup  of $H$. Then 
$$
\sqrt{[H:Z(H)]}=[H:A].
$$ 
\end{lemma}
\begin{proof} Since the commutator subgroup is cyclic then~\eqref{symplectic-module} gives a symplectic structure on $H/Z(H)$. Then the lemma is a direct consequence of Proposition, Section 2.2 of~\cite{Symplectic}.
\end{proof}
\subsection{Proof of Theorem~\ref{2-step}}   
Let $(\rho,V)$ be a faithful representation of ${H}$ with minimal dimension. Then $\rho$ decomposes completely into irreducible representations $V_i$, $1\leq i\leq r$ where $r=\rank(Z(H))=\mff(Z(H))$. By Schur's lemma, for any $z\in Z(H)$ and any $v\in V_i$ we have
$\rho_{|_{Z(H)}}(z)(v)=\chi_i(z)v$, where $\chi_i$ is a one-dimensional representation of $Z(H)$. Let $h$ be a generator of $[H,H]$ of order $p^n$. We claim that there exists an $i$ such that $\chi_i(h)$ is a primitive $p^n$th root of unity in $\mathbb{C}$. Assume the contrary. Then for all $i$, we have $\chi_{i}(h^{p^{n-1}})=1$, which implies that $h^{p^{n-1}}\in \ker\rho$. This contradicts the fact that $\rho$ is a faithful representation. Hence there exists an $i$ such that  $\ker{\chi_i}\cap[H,H]=\{\mathbf 1\}$. By Corollary~\ref{cor-SV}, we should have $V_i=\Ind_A^H(\tilde{\chi_i})$, where $\tilde{\chi_i}$ is any extension of $\chi_i$ to $A$. Therefore $\dim(V_i)=[H:A]$ and so $\mff(H)\geq [H:A]+\mff(Z(H))-1$. By Corollary~\ref{Upper-cyclic-comm} we also have $\mff(H)\leq [H:A]+\mff(Z(H))-1$. These facts together with Lemma~\ref{max-[H:A]} proves Theorem~\ref{2-step}.
\section{Representations of Heisenberg groups}\label{heisenberg}
For a (commutative and unital) ring $R$, the {\it Heisenberg group} with entries in $R$ is defined by
$$
H:=\Hei_{2k+1}(R):=\left\{ ({\bf x}, {\bf y}, z):=
\begin{pmatrix}
1 & {\bf x} & z\\
0 & I_k & {\bf y}^T\\
0 & 0 & 1
\end{pmatrix}:\quad {\bf x,y }\in R^k,\, z\in R \right\}.
$$
We record two basic identities:
\begin{equation}\label{id-com}
\begin{split}
({\bf x}_1,{\bf y}_1,z_1)({\bf x},{\bf y},z)({\bf x}_1,{\bf y}_1,z_1)^{-1}&=({\bf x},{\bf y},{\bf x}_1{\bf y}^T-{\bf x}{\bf y}_1^T+z),\\
[({\bf x}_1,{\bf y}_1,z_1),({\bf x}_{2},{\bf y}_{2},z_{2})]&=(0,0,{\bf x}_1{\bf y}_{2}^T-{\bf x}_{2}{\bf y}_1^T).
\end{split}
\end{equation}
We will also make use of the following subgroups of $H$:
\begin{equation}\label{Def-AH}
A:=\left\{
({\bf {x}}, 0, z):\, {\bf x}\in R^k,\, z \in R\right\}, \quad L:=\left\{
(0, {\bf y}, 0):\, {\bf y}\in R^k\right\}, 
\quad
Z= \left\{
(0, 0, z):\, z \in R \right\}.
\end{equation}
It is easy to see that $A$ is a maximal abelian subgroup of $H$, $H=A\rtimes L$ and $Z$ is the center of $H$. We identify the center of $\Hei_{2k+1}(R)$ with $R$ in the obvious way.
\subsection{Proof of Theorem \ref{Heisenberg}} We will need the following simple inequality later.
\begin{lemma}\label{ineq}
Let $a_0, \dots,  a_{ m -1}$ be non-negative real numbers with $\sum_{i=0}^{m-1}a_i=m$, and assume that for any $ 0  \le i \le m-1$ we have $ a_i + \cdots + a_{m-1} \le m-i$. Then for any decreasing sequence $x_0 \ge \cdots \ge x_{m-1}$, we have 
\[ \sum_{i=0}^{m-1 } a_i x_i \ge \sum_{i=0}^{m-1} x_i. \]
\end{lemma}

Now we recall briefly the {\it little group} method. For a more detailed discussion, we refer the reader to~\cite{Serre},  $\S$8.2. Let $H=\Hei_{2k+1}(\OO{n})$. Then $H=A\rtimes L$ where $A$  and $L$ are as in~\eqref{Def-AH}. Consider the action of $H$ on 
$\widehat{A}$ given by $(h\cdot\psi)(a):=\psi(h^{-1}ah)$. 
Let $(\psi_s)_{s\in \FF}$ be a system of  representatives for the $L$-orbits in $\widehat{A}$, and for each $s\in \FF$ set $L_s=\mathrm{Stab}_{L}(\psi_s)$ so that $h\cdot\psi_s=\psi_s$ for all $h\in L_s$. Set  $H_s:=A L_s$. Thus, one can show that the extension of $\psi_s$, defined by $\psi_s(ah):=\psi_s(a)$, $a\in A$ and $h\in L_s$, is a one-dimensional representation of $H_s$. Let $\lambda$ be an irreducible representation of $L_s$. We obtain
an irreducible representation $\tilde{\lambda}$ of $H_s$ by setting $\tilde{\lambda}(ah) := \lambda(h)$, $a\in A$ and $h\in L_s$. Now, $\psi_s\otimes\tilde{\lambda}$ gives an irreducible representation of $H_s$. Finally define $\theta_{s,\lambda}:=\mathrm{Ind}_{H_s}^H(\psi_s\otimes\tilde{\lambda})$. 
\begin{theorem}[Wigner-Mackey theory]~\label{Mackey} Under the above assumptions, we have 
\begin{enumerate}
\item[(i)] $\theta_{s,\lambda}$ is irreducible.
\item[(ii)] If $\theta_{s,\lambda}$ and $\theta_{s',\lambda'}$ are isomorphic, then $s=s'$ and $\lambda$ is isomorphic to $\lambda'$.
\item[(iii)] Every irreducible representation of $H$ is isomorphic to one of the $\theta_{s,\lambda}$.
\end{enumerate}
\end{theorem}
Recall that $\psi: \OO{n}\to\CC^*$ is our fixed primitive character. Then from Lemma~\ref{Character-local}, any element in the character group $\widehat{A}$ is of the form 
\begin{equation}
\psi_{{\bf b},b}({\bf x},0,z):=\psi_{\bf b}({\bf x})\psi_b(z)=\psi(b_1x_1+\dots+b_kx_k+bz),
\end{equation}
where 
${\bf b}=(b_1,\dots,b_k)\in (\OO{n})^k$, $b\in\OO{n}$ and $\psi_{\bf b}({\bf x})=\psi_{b_1}(x_1)\cdots\psi_{b_k}(x_k)$. 
\begin{proposition}\label{level} Let $\psi_{{\bf b},b}\in \widehat{A}$ with $b\in\mathfrak{p}^i/\mathfrak{p}^n$ and $b\not\in\mathfrak{p}^{i+1}/\mathfrak{p}^n$ for $0\leq i\leq n$. Then $\mathrm{Stab}_L(\psi_{{\bf b},b})$ has $q^{ik}$ elements.
\end{proposition}
\begin{proof} Let $(0,{\bf y},0)\in \mathrm{Stab}_L(\psi_{{\bf b},b})$. 
By~\eqref{id-com}, for any $({\bf x},0,z)$ we have 
\begin{equation}
(0,{\bf y},0)({\bf x},0,z)(0,{\bf y},0)^{-1}=({\bf x},0,-{\bf x}{\bf y}^T+z).
\end{equation}
Since $(0,{\bf y},0)\in \mathrm{Stab}_L(\psi_{{\bf b},b})$, we obtain
\begin{equation}
\psi_{{\bf b},b}({\bf x},0,z)=\psi_{{\bf b},b}({\bf x},0,-{\bf x}{\bf y}^T+z)\ \ \text{for every}\ ({\bf x},0,z)\in A.
\end{equation}
This means that for all ${\bf x}\in (\OO{n})^k$ we have $\psi(b{\bf x}{\bf y}^T)=1$.
Since $b\in(\mathfrak{p}^i/\mathfrak{p}^n)\setminus (\mathfrak{p}^{i+1}/\mathfrak{p}^n)$ and $\psi$ is primitive, we must have $(0,{\bf y},0)\in (\mathfrak{p}^{(n-i)}/\mathfrak{p}^n)^k$. Therefore 
$$\mathrm{Stab}_L(\psi_{{\bf b},b})\cong(\mathfrak{p}^{(n-i)}/\mathfrak{p}^n)^k,$$
 which has $q^{ik}$ elements. 
\end{proof}
\begin{definition} For an arbitrary $\ell\geq 1$ and $b=\tilde{b}+\mathfrak{p}^\ell\in \OO{\ell}$, the level of $b$ is defined to be $\min\{\nu(\tilde{b}),\ell\}$. Similarly the level of the additive character $\psi_b:\OO{\ell}\to\CC^*$ is defined by the level of $b$.
\end{definition}
Let $(\psi_{{\bf b}_s,b_s})_{s\in \FF}$ be a system of representatives  for the orbits of $L$ in $\widehat{A}$. Then the representations $\theta_{s,\lambda}$, obtained from  Theorem~\ref{Mackey}, list all irreducible representations of $H$.
\begin{remark}\label{important-remark-central} The central character of the irreducible representation $\theta_{s,\lambda}$, obtained from $\psi_{{\bf b}_s,b_s}$, is $\psi_{b_s}$.
\end{remark} 
  Since any irreducible representation of $L$ is one-dimensional, Proposition~\ref{level} allows one to compute the dimension of $\theta_{s,\lambda}$. The following corollary shows that the dimension of any irreducible representation of the Heisenberg group $H$ is determined by the level of its central character.
\begin{corollary}\label{level-dim} The dimension of $\theta_{s,\lambda}$ is $q^{(n-m_s)k}$, where $m_s$ is the level of the central character of $\theta_{s,\lambda}$. 
\end{corollary}
Let $F$ be a $p$-adic field with absolute ramification index $e$ (see Section \ref{Sec:Intro}). Then $p\mathcal{O}=\mathfrak{p}^e$ and hence for $n\geq e$ we have $\OO{n}\otimes_{\ZZ}\FZ=\mathcal{O}/\mathfrak{p}^e$.
When $n<e$ we have $\OO{n}\otimes_{\ZZ}\FZ=\OO{n}$. In both cases, $\OO{n}\otimes_{\ZZ}\FZ=\OO{\xi}$ where $\xi=\min\{n,e\}$. For the local field $\F{q}((T))$ we obviously have $\OO{n}\otimes_{\ZZ}\FZ=\OO{n}$.
\begin{lemma}\label{rank-rami} Let $F$ be a non-Archimedean local field with the  absolute ramification index $e$ and the absolute inertia degree $f$. For each $n \ge 1$, we have $\Omega_1(\OO{n})=\mathfrak{p}^{n- \xi }/\mathfrak{p}^{n}$ and $d(\OO{n})= f \xi$, 
where $\xi= \min \left\{ n, e \right\}$.
\end{lemma}
Using this we now construct a faithful representation of $\Hei_{2k+1}(\OO{n})$.
\begin{lemma}\label{cons-faith} The group $\Hei_{2k+1}(\OO{n})$ has a faithful representation of dimension 
$
\sum_{i=0}^{ \xi -1} f q^{k(n-i)},
$
where $f$ is the absolute inertia degree, $e$ is the absolute ramification index, and  $\xi=\min\{n,e\}$. 
\end{lemma}
\begin{proof}
Let $\omega_1,\dots,\omega_f$ be units in $\mathcal{O}$ such that $\{\omega_1+\mathfrak{p},\dots,\omega_f+\mathfrak{p}\}$ forms a basis for $\mathcal{O}/\mathfrak{p}$ over $\FZ$. Define
\begin{equation}\label{bij}
b_{ij}=\omega_i\varpi^j\ ,\  1\leq i\leq f,\, 0\leq j\leq \xi-1,
\end{equation}
where $\xi=\min\{n,e\}$. Set $b'_{ij}:=b_{ij}+\mathfrak{p}^n$. Thus the set $\{\psi_{b'_{ij}}\}$
contains exactly $f$ elements of the level $j$ for each $0\leq j\leq \xi-1$. Using Theorem~\ref{Mackey}, we can construct an irreducible representation $\theta_{ij}$ of $H$ with the central character $\psi_{b'_{ij}}$. Notice that $\Omega_1(\OO{n})=\mathfrak{p}^{n- \xi }/\mathfrak{p}^{n}$ and so by Lemma~\ref{Character-local}, 
\begin{equation}\label{can-iso}
\Psi: \OO{\xi}\rightarrow\widehat{\Omega}_1(\OO{n})\ ,\ b+\mathfrak{p}^\xi\mapsto \restr{{\psi_{(b+\mathfrak{p}^n)}}}{{\Omega_1(\OO{n})}}, 
\end{equation}
is a group isomorphism and hence is a $\Z{p}$-vector space isomorphism.
It is easy to verify that the set
\begin{equation}\label{cons-basis}
\left\{b_{ij}+\mathfrak{p}^\xi \ :\  1\leq i\leq f, 0\leq j\leq \xi-1\right\},
\end{equation}
is a basis for the $\Z{p}$-vector space $\OO{\xi}$, and therefore from the $\Z{p}$-isomorphism~\eqref{can-iso} we can conclude that the restrictions of $\{\psi_{b'_{ij}}: 1\leq i\leq f, 0\leq j\leq\xi-1\}$ to $\Omega_1(\OO{n})$, is a basis for $\widehat{\Omega}_1(\OO{n})$. Thus by Lemma~\ref{central-span-faith},
$$\rho:=\bigoplus_{i,j}\theta_{ij}\ ,\  1\leq i\leq f,\ 0\leq j\leq \xi-1,$$ is a faithful representation of $H$. Now by Corollary~\ref{level-dim}, we have
$
\dim\rho=\sum_{i=0}^{ \xi -1} f q^{k(n-i)}.
$
\end{proof}
Next we prove the main theorem of this section. 
\begin{proof}[ Proof of the Theorem \ref{Heisenberg}] 
Let $r:=\rank(\OO{n})=f \xi$ (see Lemma~\ref{rank-rami}).
Let $\rho$ be a faithful representation of $H$ of minimal dimension and let 
\begin{equation}
\rho=\theta_{s_1,\lambda_1}\oplus\dots\oplus\theta_{s_r,\lambda_r},
\end{equation}
be the decomposition of $\rho$ into irreducible factors.
Let $(\psi_{b'_i})_{1\leq i\leq r}$, where $b_i'=b_i+\mathfrak{p}^n\in\OO{n}$, be the central character of $\theta_{s_i,\lambda_i}$. From Lemma~\ref{Meyer}, we know that $\{\psi_{b'_i}: 1\leq i\leq r\}$, viewed as
elements of $\widehat{\Omega}_1(\OO{n})$, is a basis for $\widehat{\Omega}_1(\OO{n})$. From the $\Z{p}$-isomorphism~\eqref{can-iso}, the set
 \begin{equation}\label{basisI}
\left\{b_1+\mathfrak{p}^\xi,\dots, b_r+\mathfrak{p}^\xi \right\},
\end{equation} 
is a basis for the $\FZ$-vector space $\OO{\xi}$.
For each $0\leq i\leq \xi-1$, the dimension of $\mathfrak{p}^i/\mathfrak{p}^\xi$ over $\FZ$ is $(\xi-i)f$. Therefore the number of elements in the basis~\eqref{basisI} with level at least $i$ is at most $(\xi-i)f$. For $0 \le i \le \xi-1$, let $\alpha_i$ denote the number of elements in the set~\eqref{basisI} which are of level $i$. Then $ \alpha_i + \cdots + \alpha_{\xi-1} \le (\xi- i)f$, for all 
$0 \le i \le \xi -1$ and $  \alpha_0 + \cdots + \alpha_{\xi-1} = \xi f$. From Corollary~\ref{level-dim} we conclude that the dimension of $\rho$ is $\sum_{i=0}^{ \xi -1}  \alpha_i q^{k(n-i)}$. Now by applying Lemma~\ref{ineq} (for $a_i=\alpha_i/f$ and $m=\xi-1$), we find that the dimension of the representation $\rho$ is at least
$
 \sum_{i=0}^{ \xi -1} f q^{k(n-i)}. 
$
Lemma~\ref{cons-faith} completes the proof.
  \end{proof}
\section{Weil representation and proof of Theorem~\ref{Unipotent}}\label{Weil-Sec} This section is devoted to the proof of Theorem~\ref{Unipotent}. Set $R=\OO{n}$ and assume that $\mathrm{char}(\mathcal{O}/\mathfrak{p})\neq 2$. Let $\U_k(R)\subseteq\GL_k(R)$ denote the group of unitriangular matrices. It is easy to see that $\Hei_{2k+1}(R)$ is a normal subgroup of $\U_{k+2}(R)$ and $\U_{k+2}(R)=\Hei_{2k+1}(R)\rtimes \U_k(R)$.  We also remark that $Z(\U_{k+2}(R))=Z(\Hei_{2k+1}(R))$.
The Heisenberg group $\Hei_{2k+1}(R)$ may be realized as follows. Let $V=R^{k}\times R^{k}$ be a finite and free $R$-module. Define the symplectic form 
\begin{equation}\label{special-symp}
\langle (\x_1,\y_1),(\x_2,\y_2)\rangle=\x_1\y_2^T-\y_1\x_2^T.
\end{equation}
The Heisenberg group is defined as 
$H(V)=\{(r,v): r\in R, v\in V\}$
 with multiplication $$
(r_1,v_1)(r_2,v_2)=(r_1+r_2+\langle v_1,v_2\rangle, v_1+v_2).$$
The symplectic group 
$$\Sp=\{g\in \GL(V): \langle gv_1,gv_2\rangle=\langle v_1,v_2\rangle,\, \forall v_1,v_2\in V\},$$
 acts on the Heisenberg group $H(V)$ by $g(r,v)=(r,gv)$ for $g\in \Sp$ and $(r,v)\in H(V)$. It is easy to see that the group $H(V)$ is indeed isomorphic to $\Hei_{2k+1}(R)$. Moreover we can identify $\U_{k+2}(R)$ with a subgroup of $H(V)\rtimes \Sp$. Given an ideal  $\aideal\unlhd R$, we set
$$
V(\aideal)=\{v\in V: \langle v,V\rangle\subseteq \aideal\},
$$
and denote the quotient module $V/V(\aideal)$ by $V_{\aideal}$. Obviously $V(\aideal_1)\subseteq V(\aideal_2)$ when $\aideal_1 \subseteq \aideal_2$.
For $b\in R$, let $\psi_b$ be an additive character of $R$. The set of ideals of $R$ contained
in $\ker\psi_b$ has a unique maximal (with respect to inclusion) element $\aideal_{\psi_b}$, which is called the conductor of $\psi_b$.
\begin{lemma}\label{hei-ker} Let $b\in\OO{n}$ with the level $i$ where $0\leq i\leq n$. Then $\aideal_{\psi_b}=\mathfrak{p}^{n-i}/\mathfrak{p}^n$.
\end{lemma}
\begin{proof}
Clearly we have $\mathfrak{p}^{n-i}/\mathfrak{p}^n\subseteq \aideal_{\psi_b}$. Now let $x\in \aideal_{\psi_b}$. Then for any $s\in\OO{n}$ we have $\psi(bxs)=1$. Since the level of $b$ is $i$ and $\psi$ is primitive, we have $x\in \mathfrak{p}^{n-i}/\mathfrak{p}^n$.
\end{proof}
We are ready to state the existence of the  Schr\"{o}dinger representation.
  \begin{proposition}\label{Cliff-prop}
Let $\psi_b$ be an additive character of $R=\OO{n}$ and assume $\mathrm{char}(\mathcal{O}/\mathfrak{p})\neq 2$. Then there exists a unique
irreducible representation $\sigma_{\psi_b}$ (called the Schr\"{o}dinger representation) of $H(V)$ with central character $\psi_b$ which is $\Sp$-invariant. Its dimension is equal to $\sqrt{|V_{\aideal_{\psi_b}}|}$.
\end{proposition}
\begin{proof}
See~\cite[Proposition 2.2]{Cliff}.
\end{proof}
\begin{remark} It is important to notice that the uniqueness in the above proposition may fail if we do not impose the $\Sp$-invariance condition.   
\end{remark}
For the symplectic form~\eqref{special-symp} we can compute the dimension of the Schr\"{o}dinger representation precisely. 
\begin{lemma}\label{n-i} Let $b\in\OO{n}$ be an element of level $i$. Then the dimension of the Schr\"{o}dinger representation $\sigma_{\psi_b}$ associated to $\psi_{b}$ is $q^{k(n-i)}$.
\end{lemma}
\begin{proof} From Lemma~\ref{hei-ker} we now that $\aideal_{\psi_b}=\mathfrak{p}^{n-i}/\mathfrak{p}^n$. Then~\eqref{special-symp} shows that 
$$
V(\aideal_{\psi_{b}})=\left(\mathfrak{p}^{n-i}/\mathfrak{p}^n\right)^k\times \left(\mathfrak{p}^{n-i}/\mathfrak{p}^n\right)^k.
$$
But $|\mathfrak{p}^{n-i}/\mathfrak{p}^n|=q^{i}$ and so Proposition~\ref{Cliff-prop} completes the proof.
\end{proof}
\begin{definition} Let $\psi_b$ be an additive character of $R$ with the Schr\"{o}dinger representation $\sigma_{\psi_b} : H(V)\to \GL(X)$. A Weil representation of type $\psi_b$, is a linear representation $\tilde{\sigma}_{\psi_b}: \Sp\to \GL(X)$ such that for all $h\in H(V)$ and $g\in\Sp$
\begin{equation}\label{weil}
\tilde{\sigma}_{\psi_b}(g)\sigma_{\psi_b}(h)=\sigma_{\psi_b}(g\cdot h)\tilde{\sigma}_{\psi_b}(g).
\end{equation}
\end{definition} 
Since the Weil representation is linear, from~\eqref{weil} we can deduce that the Schr\"{o}dinger representation $\sigma_{\psi_b}$ can be extended to $H(V)\rtimes \Sp$ by mapping $(h,g)$ to $\sigma_{\psi_b}(h)\tilde{\sigma}_{\psi_b}(g)$. 
We will spend the rest of this section proving Theorem~\ref{Unipotent}.
\begin{proposition} For each additive character ${\psi_b}$ of $R$ there exists a Weil representation $\tilde{\sigma}_{\psi_b}$ of type ${\psi_b}$.
\end{proposition}
\begin{proof}
See~\cite[Theorem 3.2]{Cliff}.
\end{proof}
We are ready to prove Theorem~\ref{Unipotent}.
\begin{proof}[Proof of Theorem~\ref{Unipotent}] We will prove $\mff(\Hei_{2k+1}(\OO{n}))=\mff(\U_{k+2}(\OO{n}))$. Define
$$
b_{ij}=\omega_i\varpi^j\ ,\ 1\leq i\leq f,\ 0\leq j\leq \xi-1.
$$
Set $b'_{ij}:=b_{ij}+\mathfrak{p}^n$. Let $\sigma_{\psi_{b'_{ij}}}$ be the Schr\"{o}dinger representations associated to the additive characters $\psi_{b'_{ij}}$. Notice that the center of $H(V)$ is $\{(r,0): r\in R\}$ which we identify with $R=\OO{n}$. The proof of Lemma~\ref{cons-faith} shows that $\{\psi_{b'_{ij}}\}$ is a basis for $\widehat{\Omega}_1(\OO{n})$. Thus by Lemma~\ref{central-span-faith},
$$\rho:=\bigoplus_{i,j} \sigma_{\psi_{b'_{ij}}}\ , \ 1\leq i\leq f,\ 0\leq j\leq \xi-1,$$
 is a faithful representation of $H(V)$. Moreover, Lemma~\ref{n-i} shows that $\dim(\rho)=\mff(\Hei_{2k+1}(\OO{n}))$.
For each Schr\"{o}dinger representation $\sigma_{\psi_{b'_{ij}}}$ we have an associated Weil representation. Therefore $\sigma_{\psi_{b'_{ij}}}$ can be extended to a representation of $H(V)\rtimes\Sp$. Thus $\rho$ can be also extended to a representation $\tilde{\rho}$ of $H(V)\rtimes\Sp$. We claim that $\tilde{\rho}|_{\U_{k+2}(R)}$ is a faithful representation. Notice that $\U_{k+2}(R)$ is a $p$-group and $Z(\U_{k+2}(R))=Z(\Hei_{2k+1}(R))$. Hence $\tilde{\rho}|_{\U_{k+2}(R)}$ is faithful on the center of $\U_{k+2}(R)$ and so by Remark~\ref{fact=Remark}, $\tilde{\rho}|_{\U_{k+2}(R)}$ is a faithful representation. Therefore $\mff(\U_{k+2}(R))\leq \mff(\Hei_{2k+1}(R))$. Evidently we also have $\mff(\Hei_{2k+1}(\OO{n}))\leq \mff(\U_{k+2}(\OO{n}))$ which completes the proof. 
\end{proof}
\section{Representations of affine groups}\label{Representation of Linear groups}
In this section we record some elementary results on faithful representations of affine groups.  We first recall the character formula of the induced representation.
\begin{lemma}\label{ind-Char} Let $G$ be a finite group with a subgroup $H$. Suppose $(V, \rho)$ is induced by $(W, \theta)$ and let $\chi_\rho$ and $\chi_\theta$  be the corresponding  characters  of  $G$  and  of  $H$. Let $\mathcal{R}$ be a system of representatives of $G/H$. For each $g\in G$, we have
$$
\chi_\rho(g)=\sum_{\substack{r\in \mathcal{R}\\ r^{-1}gr\in H}}\chi_{\theta}(r^{-1}gr).
$$
\end{lemma}
\begin{proof}
See~\cite[$\S$3.3, Theorem 12]{Serre}.
\end{proof}
We also recall that for $n$ complex numbers $z_1,\dots, z_n\in \CC$ with 
$|z_{j}| \le 1$, $1 \le j \le n$, the equality $z_1+\dots+z_n=n$ forces  all $z_{j}$ to be equal to $1$. From this fact it is easy to prove the following lemma.
\begin{lemma}\label{Ker-rho} Let $G$ be a finite group and let $\rho:G\to \GL(V)$ be a representation with character $\chi$. Then 
$
\ker(\rho)=\{g\in G: \chi(g)=\chi(1)\}.
$
\end{lemma}
Using these lemmas we have:
\begin{lemma}\label{faithfulH} Let $A$ be a finite abelian group, and suppose that the finite group $H$ acts on $A$. Let $\chi$ be a one-dimensional representation of $A$ with the property that for any $0\neq a\in A$, there exists $h'\in H$ such that $\chi(h'a)\neq 1$. Then 
$\rho:=\mathrm{Ind}_A^{A\rtimes H}(\chi)$
is a faithful representation of $A\rtimes H$. In particular, if $A$ is a cyclic group then 
$\mff(A\rtimes H)\leq |H|$.
\end{lemma}
\begin{proof}
Set $G=A\rtimes H$ and identify $A$ and $H$ with their isomorphic copies inside $G$.  
Lemma~\ref{ind-Char} asserts that for any $g=(a,h)\in G$
\begin{equation}
\chi_\rho(g)=\begin{cases} 
\sum\limits_{h'\in H}\chi(h'a) & \textrm{if } \, h=1;\\
0 & \text{otherwise}.
\end{cases}
\end{equation}
For $g=(a,h)\in \ker(\rho)$ we have $\chi_\rho(g)=|H|$, which implies that $h=1$. Assume $a\neq 0$, then by our assumption there exists $h'\in H$ such that $\chi(h'a)\neq 1$ and so $\chi_\rho(g)\neq |H|$. This establishes that $\rho$ is a faithful representation. When $A$ is cyclic, we can choose $\chi$ to be a one-dimensional faithful representation of $A$ which clearly satisfies the assumption of the lemma.
\end{proof} 
\begin{lemma}\label{Lemma-primitive}
Let $C= \langle a \rangle$ be a cyclic group with $p^n$ elements, where $p$ is a prime and $n \ge 1$. Let $H$
be a finite group acting on $C$ by automorphisms and $\rho: C \rtimes H\to \GL_d(\CC)$, be a faithful representation and
$$ \restr{\rho}{C}= \bigoplus_{ \chi \in \Delta} \chi\ ,\ \Delta\subseteq \widehat{C},$$ be the decomposition of $ \restr{\rho}{C}$ to one-dimensional representations of $C$. Then, there exists $ \chi \in \Delta$ 
such that $ \chi(a)$ is a primitive $p^n$th root of unity.
\end{lemma}
\begin{proof} Note that $ \{ \1 \}=\ker \rho_{|_C}= \bigcap_{ \chi \in \Delta} \ker \chi.$
Since the lattice of subgroups of the cyclic group of order $p^n$ is totally ordered, there exists $ \chi \in \Delta$ with 
$\ker \chi= \{ \1 \}$. This implies that $ \chi(a)$ is a primitive $p^n$th root of unity. 
\end{proof}
\begin{proof}[ Proof of Proposition~\ref{Aff-prop}] Let $\rho: C\rtimes H\to\GL_d(\CC)$ be a faithful representation and let
$$ \restr{\rho}{C}= \bigoplus_{ \chi \in \Delta} \chi\ ,\ \Delta\subseteq \widehat{C},$$
be the decomposition of $ \restr{\rho}{C}$ into one-dimensional representations of $C$.
Let $h_ia=a^{m_{i}}$, $1\leq i\leq l$, be the $H$-orbit of the generator $a$. By Lemma~\ref{Lemma-primitive}, there exists $\chi\in\Delta$ such that $\zeta= \chi(a)$ is a primitive $p^n$th root of unity. Since $H$ acts on $C$, the set $ \Delta$ is also $H$-invariant. Hence for each $1\leq i\leq l$, we obtain a one-dimensional representation $\chi_{h_i}\in \Delta$ such that 
$$
 \chi_{h_i}(a)=\chi(h_ia)=\chi(a^{m_i})=\zeta^{m_i}.
$$  
These representations are clearly all distinct and hence $\dim(\rho)\geq |\Delta|\geq 
|Ha|$. This establishes the first part of the proposition. If $C$ is a faithful $H$-module, then $|Ha|= |H|$, and so by Lemma~\ref{faithfulH} we have the second part of the proposition.
\end{proof}

Let us remark that the proof of Proposition~\ref{Aff-prop} relies on the fact that $\Z{p^n}$ is a cyclic group, which no longer holds for general $\OO{n}$. We will get around this issue by analyzing the characters of $\OO{n}$. 
\begin{proof}[ Proof of Theorem~\ref{Aff}] Let $\rho: \Aff(\OO{n})\to \GL_d(\CC)$ be a faithful representation and consider the following decomposition
$$
\restr{\rho}{(\OO{n})}=\bigoplus_{\psi_i\in\Delta}\psi_i\ ,\ \Delta\subseteq \widehat{\OO{n}}.
$$
We claim that there exists $i$, such that $\psi_{i}\in\Delta$ is primitive. This is true since the lattice of ideals of $\OO{n}$ is totally ordered and $\rho$ is a faithful representation. This establishes the claim. Denote this primitive character by $\psi$. Notice that $\Delta$ is $(\OO{n})^\times$-invariant. Indeed for each $b\in(\OO{n})^\times$ we obtain a new primitive character $\psi_b\in\Delta$, defined by $\psi_b(x)=\psi(bx)$. Therefore $|\Delta|\geq |(\OO{n})^\times|=q^{n}-q^{n-1}$ which implies that
$$
\mff(\Aff(\OO{n}))\geq q^{n}-q^{n-1}.
$$
We now construct a faithful representation of dimension $q^n-q^{n-1}$. The idea of the construction 
resembles the one used in Lemma~\ref{faithfulH}. Let $\psi$ be a primitive character of $\OO{n}$ and let $0\neq x\in\OO{n}$ be an arbitrary element. We claim that there is a unit $t\in\OO{n}$ such that $\psi(tx)\neq 1$. Since $\psi$ is primitive, there exists $s\in\OO{n}$ such that $\psi(sx)\neq 1$. Suppose that $s$ is not unit. If $\psi(x)=1$, then $t=1+s$ is the desired unit and if $\psi(x)\neq 1$ then we can take $t=1$.  Therefore by Lemma~\ref{faithfulH} we observe that $\mathrm{Ind}_{\OO{n}}^{\Aff(\OO{n})}(\psi)$ is a faithful representation of dimension $q^n-q^{n-1}$.
\end{proof}
\section*{Acknowledgement} During the completion of this work, M.B. was supported by a postdoctoral fellowship from the University of Ottawa. He wishes to thank
his supervisors Vadim Kaimanovich, Hadi Salmasian, and Kirill Zainoulline. The authors thank  Camelia Karimianpour, Monica Nevins, and Zinovy Reichstein  for useful discussions on the subject of this paper. The authors are also very grateful to the referee for a very careful reading and many helpful suggestions and corrections.
\bibliographystyle{abbrv}

\end{document}